\documentclass{article}
\usepackage{amsfonts}
\usepackage{amssymb}
\usepackage{amsmath}

\newtheorem{theorem}{Theorem}

\newtheorem{corollary}[theorem]{Corollary}

\newtheorem{lemma}[theorem]{Lemma}

\newtheorem{remark}[theorem]{Remark}

\newenvironment{proof}[1][Proof]{\textbf{#1.} }{\ \rule{0.5em}{0.5em}}
\begin{document}

\title{\textbf{Mean-Square Continuity on Homogeneous Spaces of Compact Groups%
\thanks{%
We thank S. Trapani for a crucial input to the content of Section \ref%
{s:homs}; we are also grateful to P.Baldi, O. El-Dakkak, N.Leonenko and
H.-J.Starkloff for useful discussions. Research by DM\ is supported by the
ERC Grant 277742 Pascal.}}}
\author{D.Marinucci and G.Peccati \\
%EndAName
Department of Mathematics, University of Rome Tor Vergata\\
and Luxembourg University, Unit\'{e} de Recherche en Math\'{e}matiques}
\maketitle

\begin{abstract}
We show that any finite-variance, isotropic random field on a compact group
is necessarily mean-square continuous, under standard measurability
assumptions. The result extends to isotropic random fields defined on
homogeneous spaces where the group acts continuously.

\begin{itemize}
\item \textbf{Keywords and Phrases: }Random Processes; Isotropy; Mean-Square
Continuity.

\item \textbf{AMS Classification: }60G05, 60G60\textbf{\ }
\end{itemize}
\end{abstract}

\section{Introduction}

The analysis of the spectral representations of stationary and isotropic
finite-variance random fields on (subsets of) $\mathbb{R}^{d}$ is a
classical topic of probability theory, presented in many standard textbooks
in the area (see for instance \cite{Adler, adlertaylor, Leon, Yad}). Such
representations are sometimes based on Karhunen-Lo\`eve constructions
(specially under Gaussianity assumptions), realized by first computing the
eigenfunctions associated with the covariance kernel, and then by expanding
the field into these orthogonal components (see for instance \cite%
{adlertaylor}). In other cases, the argument proceeds from the construction
of an isometry between an $L^{2}$ space of deterministic square integrable
functions, and some space of finite-variance random variables, with inner
product defined in terms of the covariance function of the process to be
represented (see for instance \cite{Leon, Yad}). In all these approaches,
mean-square continuity is assumed as a necessary condition to ensure that
the spectral representation holds pointwise.

More recently, considerable attention has been drawn to the case where the
process at hand is defined on the homogenous space of a compact group
(including the group itself). In this context, one of the most relevant
examples for applications is the sphere $S^{2},$ which is well-known to be
isomorphic to the quotient space $SO(3)/SO(2),$ where $SO(d)$ denotes as
usual the special group of rotations in $\mathbb{R}^{d}.$ Under these
circumstances, spectral representation results take a particularly neat
form, as they can be viewed as stochastic versions of the celebrated
Peter-Weyl Theorem (see \cite[Section 4.6]{DuiKolk} or \cite[Section 2.5]%
{MarPecBook}); the latter ensures that the matrix coefficients of the
irreducible representations of a compact group $G$ provide an orthonormal
basis for the space $L^{2}(G)$ of square-integrable functions on the group,
endowed with the Haar measure. This is the standpoint adopted for instance
by \cite{MarPecBook, PePy} -- see also \cite{BMV}. It should be noted that
random fields on the unit sphere $S^{2}$ play now an extremely important
role in many applied fields, for instance in Cosmology -- see \cite%
{dodelson, durrer}, \cite{MarPecBook} for an overview.

As we shall point out in the sections to follow, the argument based on the
group-theoretic point of view does not only provide an alternative proof for
classical results, but yields also an unexpected bonus: the assumption of
mean-square continuity turns out to be no longer necessary for the spectral
representation to hold. More than that, mean-square continuity follows as a
necessary consequence of the spectral representation, under standard
measurability conditions. The aim of this short note is then to highlight
this result, which we can state concisely as follows:

\begin{quote}
\textsl{Let $T$ be a measurable finite-variance isotropic random field
defined on the homogeneous space of a compact group acting continuously.
Then, $T$ is necessarily mean-square continuous.}
\end{quote}

Our main findings are contained in the statements of Theorem \ref{t:msc}
(for fields defined on compact groups) and of Theorem \ref{t:msch} (for
fields defined on homogeneous spaces). Section 2 and Section 3 contain
preliminary results, respectively on group representations and random
fields. Some historical remarks are provided in Section 6.

\section{Preliminaries on group representations}

We now provide a brief overview of the results about group representations
that are used in this note. For any unexplained definition or result, the
reader is referred to \cite[Chapter 2]{MarPecBook}, as well as to the
classic reference \cite{DuiKolk}. \medskip

A \textit{topological group }is a pair $\left( G,\mathbb{G}\right) $, where $%
G$ is a group and $\mathbb{G}$ is a topology such that the following three
conditions are satisfied: (i) $G$ is a Hausdorff topological space, (ii) the
multiplication $G\times G\to G:\left( g,h\right) \mapsto gh$ is continuous,
(iii) the inversion $G\to G:g\mapsto g^{-1}$ is continuous. In what follows,
we use the symbol $G$ to denote a topological group (the topology $\mathbb{G}
$ being implicitly defined) which is also \textit{compact } and such that $%
\mathbb{G}$ has a countable basis. We will denote by $C\left( G\right) $ the
class of continuous, complex-valued functions on $G$; $\mathcal{G}$ is the
(Borel) $\sigma $-field generated by $\mathbb{G}$; $dg$ denotes the
normalized \textit{Haar measure} of $G$ (in particular, $\int_G dg = 1$). We
shall denote by $L^{2}\left( G,\mathcal{G}, dg\right) =L^{2}\left( G\right) $
the Hilbert space of complex-valued functions on $G$ that are
square-integrable with respect to $dg$. Plainly, the space $L^{2}\left(
G\right) $ is endowed with the usual inner product $\left\langle
f_{1},f_{2}\right\rangle _{G}=\int_{G}f_{1}\left( g\right) \overline{%
f_{2}\left( g\right) }dg$; also, $\left\Vert \cdot \right\Vert _{G} =
\left\langle \cdot ,\cdot \right\rangle^{1/2} _{G}$. \medskip

Let $X$ be a topological space. A \textit{continuous left action} of $G$ on $%
X$ is a jointly measurable mapping $A : G\times X : (g,x) \mapsto A(g,x) :=
g\cdot x$, satisfying the following properties: (i) $g\cdot (h\cdot x) =
(gh)\cdot x$, (ii) if $e$ is the identity of $G$, then $e\cdot x = x$ for
every $x\in X$, and (iii) the mapping $(g,x)\to g\cdot x$ is jointly
continuous. Right actions are defined analogously and will not be directly
considered in this note, albeit every result concerning left actions proved
below extends trivially to right actions. The space $X$ is called a $G$-%
\textit{homogeneous space} if $G$ acts transitively on $X$, that is: for
every $y,x\in X$, there exists $g\in G$ such that $y = g\cdot x$. Group
representations (as described in the next paragraph) are distinguished
examples of group representations.

\medskip

Let $V$ be a normed finite-dimensional vector space over $\mathbb{C}$. A
(finite-dimensional) \textit{representation }of $G$ in $V$ is an
homomorphism $\pi $, from $G$ into $\mathbf{GL}\left( V\right) $ (the set of
complex isomorphisms of $V$ into itself), such that the mapping $G\times V
\to V:\left( g,v\right) \mapsto \pi \left( g\right) (v) $ is continuous.
Using e.g. \cite[Proposition 2.25]{MarPecBook}, one sees that it is always
possible to endow $V$ with an inner product $\langle \cdot, \cdot\rangle_V$
such that $\pi$ is \textit{unitary} with respect to it, that is: for every $%
g\in G$ and every $u,v\in V$, $\langle \pi(g)u, \pi(g)v\rangle_V = \langle
u, v\rangle_V$. Note that $\langle \cdot, \cdot\rangle_V$ can be chosen in
such a way that the associated norm preserves the topology of $V$ (see \cite[%
Corollary 4.2.2]{DuiKolk}). The \textit{dimension }$d_{\pi }$ of a
representation $\pi $ is defined to be the dimension of $V$. A
representation $\pi $ of $G$ in $V$\textbf{\ }is \textit{irreducible}, if
the only closed $\pi \left( G\right) $-invariant subspaces of $V$ are $%
\left\{ 0\right\} $ and $V$. It is well-known that unitary irreducible
representations are defined up to equivalence classes (see \cite[p. 25]%
{MarPecBook}). We will denote by $\left[ \pi \right] $ the equivalence class
of a given unitary irreducible representation $\pi $; the set of equivalence
classes of unitary irreducible representations of $G$ is written $\widehat{G}
$, and it is called the \textit{dual }of $G$. We recall that, according e.g.
to \cite[Theorem 4.3.4 (v)]{DuiKolk}, since $G$ is second countable (and
therefore metrizable) $\widehat{G}$ is necessarily countable.

\medskip

To every $\left[ \pi \right] \in \widehat{G}$ we associate a
finite-dimensional subspace $M_{\pi }\subseteq L^{2}\left( G\right) $ in the
following way. Select an element $\pi :G\mapsto \mathbf{GL}\left( V\right) $
in $\left[ \pi \right] $, as well as an orthonormal basis $\mathbf{e}%
=\left\{ e_{1},...,e_{d_\pi}\right\} $ of $V$. The space $M_{\pi }$ is
defined as the finite-dimensional complex vector space spanned by the
functions 
\begin{equation*}
g\mapsto \pi_{i,j} \left( g\right):= \langle \pi(g)e_j, e_i\rangle_V, \quad
i,j=1,...,d_\pi \text{.}
\end{equation*}
Note that such a definition is well given, since $M_{\pi }$ does not depend
on the choices of the representative element of $\left[ \pi \right] $ and of
the orthonormal basis of $V$. The following three facts are relevant for the
subsequent discussion: (i) $\{\sqrt{d_\pi}\pi_{ij} : i,j=1,...,d_\pi\}$ is
an orthonormal system of $L^2(G)$ (see \cite[p. 34]{MarPecBook}); (ii) $\dim
M_{\pi }=d_{\pi }^{2}$, and (iii) $M_{\pi }\subseteq C\left( G\right) $, for
every $\left[ \pi \right] \in \widehat{G}$.

\medskip

To conclude, we recall that, if two representations $\pi$ and $\pi^{\prime }$
are not equivalent, then $M_\pi$ and $M_{\pi^{\prime }}$ are orthogonal in $%
L^2(G)$. One crucial element of our discussion is the celebrated \textit{%
Peter-Weyl Theorem} (see \cite[Section 2.5]{MarPecBook}), stating that 
\begin{equation}  \label{e:pw}
L^2(G) = \bigoplus_{[\pi]\in \widehat{G}} M_\pi,
\end{equation}
that is: the family of finite-dimensional spaces $\{M_\pi : [\pi]\in 
\widehat{G}\}$ constitutes an orthogonal decomposition of $L^2(G)$. In
particular, the class $\{\sqrt{d_\pi}\pi_{ij} : [\pi]\in \widehat{G}, \,
i,j=1,...,d_\pi\}$ is an orthonormal basis of $L^2(G)$. Plainly, the
orthogonal projection of a given function $f\in L^2(G)$ on the space $M_\pi$
is given by the mapping 
\begin{equation}  \label{e:op}
g\mapsto f^\pi(g) = \sum_{i,j=1}^{d_\pi} d_\pi \int_G f(h) \bar{\pi}%
_{i,j}(h)dh \times \pi_{i,j}(g) = d_\pi \int_G f(h)\chi_{\pi}(h^{-1}g)dh,
\end{equation}
where $\chi_\pi(g) = \mathrm{Trace}\, \pi(g)$ is the \textit{character} of $%
\pi$. See \cite[Section 2.4.5]{MarPecBook} for a discussion of the basic
properties of group characters; in particular, one has that two equivalent
representations have the same character, in such a way that the projection
formula (\ref{e:op}) is well-defined, in the sense that it does not depend
on the choice of the representative element of the equivalence class $[\pi]$%
. \medskip

From now on, we shall fix a topological compact group $(G, \mathbb{G})$, and
freely use the notation and terminology introduced above.

\section{General setting and spectral decompositions}

\label{s:pw}

Let $T = \left\{ T(g) : \text{ }g\in G\right\} $ be a finite variance,
isotropic random field on $G$, by which we mean that $T$ is a real-valued
random mapping on $G$ verifying the following properties \textbf{(a)}--%
\textbf{(c)}.

\begin{itemize}
\item[\textbf{(a)}] (\textit{Joint measurability}) The field $T$ is defined
on a probability space $(\Omega ,\mathcal{F} ,P)$, and the mapping $T:G
\times \Omega \rightarrow \mathbb{R} :(g,\omega )\mapsto T(g,\omega )$ is $%
\mathcal{G} \otimes \mathcal{F} $-measurable, where (as before) $\mathcal{G}$
denotes the Borel $\sigma$-field associated with $(G, \mathbb{G})$.

\item[\textbf{(b)}] (\textit{Isotropy}) The distribution of $T$ is invariant
in law with respect to the action of $G$ on itself. This means that, for
every $h\in G$, $T(hg)\overset{d}{=}T(g),$ where \textquotedblleft $\overset{%
d}{=}$\textquotedblright\ indicates equality in distribution in the sense of
stochastic processes, that is: for every $d\geq 1$ and every $%
g_{1},...,g_{d}\in G$, the two vectors $(T(g_{1}),...,T(g_{d}))$ and $%
(T(hg_{1}),...,T(hg_{d}))$ have the same distribution. Note that, since the
application $g\mapsto hg$ is continuous, then the mapping $(g,\omega)
\mapsto T(hg,\omega)$ is jointly measurable, in the sense of Point \textbf{%
(a)} above.

\item[\textbf{(c)}] (\textit{Finite variance}) The field $T$ has finite
variance, i.e.: $ET^{2}(g)=\int T^{2}(g,\omega )dP(\omega )<\infty ,$ for
every $g\in G.$ For notational simplicity, and without loss of generality,
we will assume in the sequel that $ET(g)=0.$
\end{itemize}

Under the previous assumptions and by virtue of the invariance properties of
Haar measures, one has that, for every fixed $g_{0}\in G$, $E[T^{2}(g_{0})]=E%
\left[ \int_{G}T^{2}(g)dg\right] <\infty $. This implies that there exists a 
$\mathcal{F}$-measurable set $\Omega ^{\prime }$ of $P$-probability 1 such
that, for every $\omega \in \Omega ^{\prime }$, the mapping $T(\cdot ,\omega
):g\mapsto T(g,\omega )$ is an element of $L^{2}(G)$. For every $[\pi ]\in 
\widehat{G}$, we now define the quantity $T^{\pi }(g,\omega )$ according to (%
\ref{e:op}), whenever $\omega \in \Omega ^{\prime }$, and we set $T^{\pi
}(g,\omega )=0$ otherwise. It is easily checked that, for every $[\pi ]\in 
\widehat{G}$, the mapping $(g,\omega )\mapsto T^{\pi }(g,\omega )$ is $%
\mathcal{G}\otimes \mathcal{F}$ measurable.

\medskip

According to the results discussed in \cite[Section 5.2.1]{MarPecBook}, the
following two facts take place.

\begin{itemize}
\item[--] The class $\{T,T^{\pi} : [\pi] \in \widehat{G}\}$ is an isotropic
(possibly infinite-dimensional) square-integrable centered field over $G$,
that is: every $T^{\pi}$ is centered and square-integrable, and for every $%
m,d\geq 1$, for every $[\pi_1],...,[\pi_m]\in \widehat{G}$ and every $h,
g_1,...,g_d\in G$, the $(m+1)d$-dimensional vector 
\begin{equation*}
\big(T(g_1),...,T(g_d); \, T^{\pi_i}(g_1),...,T^{\pi_i}(g_d), i=1,...,m\big)
\end{equation*}
has the same distribution as 
\begin{equation*}
\big(T(hg_1),...,T(hg_d); \, T^{\pi_i}(hg_1),...,T^{\pi_i}(hg_d), i=1,...,m%
\big)
\end{equation*}

\item[--] Let $\{[\pi_k] :k=1,2,...\}$ be any enumeration of $\widehat{G}$.
Then, for every fixed $g_0\in G$ one has that

\begin{equation}
\lim_{n\rightarrow \infty }E\left\{\left\vert
T(g_0)\!-\!\sum_{k=1}^{n}T^{\pi_k}(g_0)\right\vert^2 \right\} =
\lim_{n\rightarrow \infty }E\left\{ \int_{G}\left\vert
T(g)-\sum_{k=1}^{n}T^{\pi_k}(g) \right\vert ^{2}dg\right\} =0\text{,}
\label{e:convg}
\end{equation}
in other words: the sequence $\sum_{k=1}^{n}T^{\pi_k}$, $n\geq 1$,
approximates $T$ in the $L^2(P)$ sense, both for every fixed element of $G$,
and in the sense of the space $L^2(G)$. Note that the equality in formula (%
\ref{e:convg}) is a consequence of the invariance and finiteness properties
of the Haar measure $dg$ and of the isotropy of $\{T,T^{\pi} : [\pi]\in 
\widehat{G}\}$, yielding

\begin{equation*}
E\left\{ \left\vert T(g_{0})\!-\!\sum_{k=1}^{n}T^{\pi
_{k}}(g_{0})\right\vert ^{2}\right\} \!=\!E\left\{ \int_{G}\left\vert
T(gg_{0})\!-\!\sum_{k=1}^{n}T^{\pi _{k}}(gg_{0})\right\vert
^{2}\!\!\!dg\right\} \!=\!E\left\{ \int_{G}\left\vert
T(g)\!-\!\sum_{k=1}^{n}T^{\pi _{k}}(g)\right\vert ^{2}\!\!\!dg\right\} .
\end{equation*}
\end{itemize}

The proof of the isotropy of $\{T^{\pi }:[\pi ]\in \widehat{G}\}$ is given
in \cite{MarPecBook}, see Proposition 5.3 on pages 116-117. It should be
noted that the proof of this Proposition implicitly exploits the fact that,
under isotropy, for every $h\in G$ the scalar products of $T(\cdot )$ and $%
T^{h}(\cdot ):=T(h\cdot )$ with any continuous function have necessarily the
same distribution (we thank P. Baldi for raising this point). This result is
trivial under mean-square continuity, but in the present general
circumstances it is a bit less obvious, and hence we report here a proof for
completeness.

\begin{lemma}
Let $T$ be an a.s. square-integrable invariant random field on $\mathcal{G}$
and define, for $f\in L^{2}(\mathcal{X})$, 
\begin{equation}
T(f):=\int_{G}T(x)\overline{f(x)}\,dx  \label{gen-coeff}
\end{equation}%
Then, for every $h\in G$ and every $f_{1},\dots ,f_{m}\in L^{2}(G)$, the two
random variables 
\begin{equation*}
(T(f_{1}),\dots ,T(f_{m}))\text{ and }(T^{h}(f_{1}),\dots ,T^{h}(f_{m}))
\end{equation*}%
have the same distribution.
\end{lemma}

\begin{proof}
For the sake of simplicity, we shall only deal with the case $m=1$; the general case follows along similar lines. For every $n>0$ let $T_{n}=T\wedge n\vee (-n)$,
which is a bounded random field, itself invariant. Now, if $f\in L^{2}(%
{G})$, by dominated convergence, 
\begin{equation*}
T_{n}(f)\overset{a.s.}{\underset{n\rightarrow \infty }{\rightarrow }}%
T(f)\quad \text{and }\quad T_{n}^{g}(f)\overset{a.s.}{\underset{n\rightarrow
\infty }{\rightarrow }}T^{g}(f)\ ,
\end{equation*}%
and therefore the convergence takes place also in distribution. Hence it is
sufficient to prove the statement under the additional assumption that $T$
is bounded. But in this case the two r.v.'s $T(f)$, $T^{h}(f)$ are
themselves bounded and in order to prove that they are equi-distributed it
is sufficient to show that they have the same moments, i.e. that, for every
choice of integers $p,q\geq 0$, 
\begin{equation*}
E[({\rm Re}T(f))^{p}({\rm Im}T(f))^{q}]=E[({\rm Re}T^{h}(f))^{p}({\rm Im}%
T^{h}(f))^{q}]\ .
\end{equation*}%
Now 
\begin{equation*}
E[({\rm Re}T(f))^{p}({\rm Im}T(f))^{q}]
\end{equation*}%
\begin{equation*}
=E\left( \int_{G}{\rm Re}(T(x_{1})\overline{f(x_{1})}%
)dx_{1}...\int_{G}{\rm Re}(T(x_{p})\overline{f(x_{p})})\int_{%
G}{\rm Im}(T(y_{1})\overline{f(y_{1})})dy_{1}...\int_{G}%
{\rm Re}(T(y_{q})\overline{f(y_{q})})dy_{q}\right) \text{ .}
\end{equation*}%
Developing the real and imaginary parts and applying Fubini's theorem, the
previous expectation reduces to a sum of terms of the form 
\begin{equation*}
\int_{G}\dots \!\!\int_{G}E\left[ b_{1}(x_{1})\dots
b_{p}(x_{p})c_{1}(y_{1})\dots c_{q}(y_{q})\right] d_{1}(x_{1})\dots
d_{p}(x_{p})e_{1}(y_{1})\dots e_{q}(y_{q})\,dx_{1}\dots dx_{m}dy_{1}\dots
dy_{k}
\end{equation*}%
where $b_{i}(x_{i})$ (resp. $c_{j}(y_{j})$) can be equal to ${\rm Re}%
T(x_{i})$ or ${\rm Im}T(x_{i})$ (resp. ${\rm Re}T(y_{j})$ or ${\rm Im}%
T(y_{j})$) and $d_{i}(x_{i})$ (resp. $e_{j}(y_{j})$) can be equal to $\pm 
{\rm Re}f(x_{i})$ or $\pm {\rm Im}f(x_{i})$ (resp $\pm {\rm Re}f(y_{j})$
or $\pm {\rm Im}f(y_{j})$). Now just remark that the quantity $E\left[
b_{1}(x_{1})\dots b_{p}(x_{p})c_{1}(y_{1})\dots c_{q}(y_{q})\right] $ does
not change if the random field $T$ is replaced by its rotated version $T^{h}$, as
such an expectation only depends on the joint distribution; this concludes the
proof.
\end{proof}

\section{Mean-square continuity on $G$}

In the next statement we show that isotropic fields such as those of the
previous section are necessarily mean-square continuous.

\begin{theorem}
\label{t:msc}Let $T$ be a square-integrable centered isotropic field on the
topological compact group $G$, verifying properties \textrm{\textbf{(a)}--%
\textbf{(c)}} of Section \ref{s:pw}. Then, $T$ is mean-square continuous:
for every $g\in G$ 
\begin{equation}  \label{e:msc}
\lim_{h\to g} E\left| T(g) - T(h)\right|^2 = 0.
\end{equation}
\end{theorem}

\begin{remark}
\textrm{Since $(G,\mathbb{G})$ is a topological space, equation (\ref{e:msc}%
) has to be interpreted in the following sense: for every net $%
\{h_i\}\subset G$ converging to $g$, one has that $E\left| T(g) -
T(h_i)\right|^2 \to 0$. See \cite[p. 28ff]{Dudley book} for details on these
notions. }
\end{remark}

\noindent\textbf{Proof of Theorem \ref{t:msc}}. Let $\{[\pi_k] :k\geq 1\}$
be an arbitrary enumeration of $\widehat{G}$. In view of (\ref{e:convg}),
for every $\epsilon >0$, there exists $n\geq 1$ such that 
\begin{equation*}
\sup_{g\in G} E\left\{\left\vert
T(g)\!-\!\sum_{k=1}^{n}T^{\pi_k}(g)\right\vert^2 \right\}\leq \frac{\epsilon%
}{6},
\end{equation*}
in such a way that, for every $h,g\in G$, 
\begin{eqnarray*}
E\left| T(g) - T(h)\right|^2 &\leq& 3\left\{ E\left[\left\vert
T(g)\!-\!\sum_{k=1}^{n}T^{\pi_k}(g)\right\vert^2\right]+E\left[\left\vert
T(h)\!-\!\sum_{k=1}^{n}T^{\pi_k}(h)\right\vert^2 \right] \right\} \\
&& + 3E\left[\left\vert\sum_{k=1}^{n}(T^{\pi_k}(g) -T^{\pi_k}(h))
\right\vert^2\right] \\
&&\leq \epsilon +3E\left[\left\vert\sum_{k=1}^{n}(T^{\pi_k}(g)
-T^{\pi_k}(h)) \right\vert^2\right] = \epsilon +3\sum_{k=1}^n E\left[%
\left\vert T^{\pi_k}(g) -T^{\pi_k}(h) \right\vert^2\right],
\end{eqnarray*}
where in the last relation we have used the fact that, for $k\neq k^{\prime
} $, the two fields $T^{\pi_k}$ and $T^{\pi_{k^{\prime }}}$ are uncorrelated
(see \cite[Proposition 5.4]{MarPecBook}). Now, for every $[\pi]\in \widehat{G%
}$, we define $\widehat{T}_{i,j}^\pi := \int_G T(h) \bar{\pi}_{i,j}(h)dh$
(it is easily seen that $\widehat{T}_{i,j}^\pi$ is a square-integrable
random variable). Using (\ref{e:op}), we therefore deduce that, for every $%
k\geq 1$, 
\begin{equation*}
E\left[\left\vert T^{\pi_k}(g) -T^{\pi_k}(h) \right\vert^2\right]\leq
d_{\pi_k}^2 \sum_{i,j=1}^{d_{\pi_k}} E[|\widehat{T}_{i,j}^{\pi_k}|^2] \times
|\pi_{i,j}(g) - \pi_{i,j}(h)|^2.
\end{equation*}
Since $\pi_{i,j}\in C(G)$, this last relation together with the previous
estimates implies that 
\begin{equation*}
\limsup_{h\to g} E\left| T(g) - T(h)\right|^2\leq \epsilon,
\end{equation*}
and the conclusion follows from the fact that $\epsilon$ is arbitrary. $%
\blacksquare$

\section{Mean-square continuity on homogeneous spaces}

\label{s:homs}

We now fix a topological compact group $G$, and consider a topological space 
$X$ which is also a $G$-homogeneous space, where $G$ acts transitively and
continuously from the left. Let $T = \{T(x) : x\in X\}$ be a centered,
finite-variance isotropic random field on $X$. This means that the following
three properties are verified: \textbf{(i)} the field $T$ is defined on a
probability space $(\Omega ,\mathcal{F} ,P)$, and the mapping $T:X \times
\Omega \rightarrow \mathbb{R} :(x,\omega )\mapsto T(x,\omega )$ is $\mathcal{%
X} \otimes \mathcal{F} $-measurable, with $\mathcal{X}$ denoting the Borel $%
\sigma$-field associated with $X$; \textbf{(ii)} for every $h\in G$, $T(hx)%
\overset{d}{=}T(x),$ where \textquotedblleft $\overset{d}{=}$%
\textquotedblright\ indicates as before equality in distribution in the
sense of stochastic processes; \textbf{(iii)} $ET^{2}(x)=\int T^{2}(x,\omega
)dP(\omega )<\infty$ and $ET(x)=0$, for every $x\in X$. Plainly isotropic
fields as those introduced in Section 3 are special cases of the above
class, obtained by taking $X=G$. To simplify the discussion, in what follows
we implicitly assume that both $X$ and $G$ are metric spaces. The following
result shows that the content of Theorem \ref{t:msc} extends to random
fields defined on $X$.

\begin{theorem}
\label{t:msch} Let $T$ be a square-integrable centered isotropic field on
the $G$-homogeneous space $X$, verifying properties \textrm{\textbf{(i)}--%
\textbf{(iii)}} above. Then, $T$ is mean-square continuous: for every $x\in
X $, 
\begin{equation}  \label{e:msch}
\lim_{y\to x} E\left| T(y) - T(x)\right|^2 = 0.
\end{equation}
\end{theorem}

The proof of Theorem \ref{t:msch} is based on the following lemma.

\begin{lemma}
\label{l:trapani} Let $\left\{ x_{n} \right\}\subset X$ be a sequence
converging to $x_{0}\in X$ in the topology of the homogeneous space (written 
$x_{n}\rightarrow _{X}x_{0}$). Then, there exists a subsequence $\left\{
x^{\prime }_{n} \right\}\subset \left\{ x_{n} \right\}$ verifying the
following property: there exists a sequence $\left\{ g_{n} \right\}\subset G 
$ such that $g_{n}\cdot x_{0}=x_{n}^{\prime }$ for every $n$, and $%
g_{n}\rightarrow _{G} e,$ where $e$ denotes the identity element of $G$.
\end{lemma}

\noindent

\begin{proof}
By transitivity, there exists a sequence $\{g_{n}^{\prime \prime }\}$ such
that $g_{n}^{\prime \prime }\cdot x_{0}\equiv x_{n}$ for every $n$.
Moreover, because the group is compact, this sequence admits a subsequence $%
\left\{ g_{n}^{\prime }\right\} \subset \left\{ g_{n}^{\prime \prime
}\right\} $ such that $g_{n}^{\prime }\rightarrow _{G}g_{0}.$ Note that $%
g_{0}$ need not be the identity (otherwise the proof would be finished), but
it does need to belong to the isotropy group of $x_{0}$, written $g_{0}\in
I\!so(x_{0})$, meaning that $g_0\cdot x_0 = x_0$. Write $x^{\prime }_n =
g^{\prime }_n\cdot x_0$. We claim that there exists a sequence $\{h_{n}\}
\subset G$ such that the sequence $g_{n}:=h_{n}g_{n}^{\prime }$, $n\geq 1$,
satisfies the following two properties:

\begin{equation*}
\mathbf{(A)}: \quad \text{ }g_{n}\cdot x_{0}=h_{n}g_{n}^{\prime }\cdot
x_{0}=g_{n}^{\prime }\cdot x_{0}=x_{n}^{\prime }\text{ (that is, }%
h_{n}g_{n}^{\prime }\in I\!so(x_{n}^{\prime }))
\end{equation*}%
and%
\begin{equation*}
\mathbf{(B)}: \quad \text{ }g_{n}=h_{n}g_{n}^{\prime }\rightarrow _{G}e\text{
.}
\end{equation*}%
Such a sequence is given by $h_{n}:=g_{n}^{\prime }g_{0}^{-1}(g_{n}^{\prime
})^{-1}.$ Indeed, it is immediate to see that%
\begin{eqnarray*}
h_{n}g_{n}^{\prime }\cdot x_{0} &=&g_{n}^{\prime }g_{0}^{-1}(g_{n}^{\prime
})^{-1}g_{n}^{\prime }\cdot x_{0}=g_{n}^{\prime }g_{0}^{-1}\cdot x_{0} \\
&=&g_{n}^{\prime }\cdot x_{0}=x^{\prime }_{n}\text{ ,}
\end{eqnarray*}%
where we have exploited the trivial fact that $g_{0}^{-1}$ is an element of $%
I\!so(x_{n}),$ because $g_{0}$ is. Hence, $\mathbf{(A)}$ is fulfilled.
Moreover, since $g_{n}^{\prime }\rightarrow _{{G}}g_{0},$ by continuity one
infers that%
\begin{equation*}
h_{n}\rightarrow _{{G}}g_{0}^{-1}, \text{ and consequently }%
h_{n}g_{n}^{\prime }\rightarrow _{{G}}e\text{,}
\end{equation*}%
yielding \textbf{(B)}. It follows that the sequence $\{x^{\prime }_n\}$
satisfies the requirements in the statement, and the proof is concluded.
\end{proof}

\medskip

\noindent\textbf{Proof of Theorem \ref{t:msch}} Fix $x\in X$ and let $x_n
\to_X x$. Using the assumptions on $T$, one infers that the mapping $g
\mapsto T(g\cdot x) := T_x(g)$ is a centered finite-variance isotropic field
on $G$, in the sense of Section 3. According to Lemma \ref{l:trapani}, there
exist sequences $\{x^{\prime }_n\}\subset \{x_n\}$ and $\{g_n\}\subset G$
such that $x^{\prime }_n =g_n x$ and $g_n \to_G e$. By virtue of Theorem \ref%
{t:msc}, we therefore conclude that 
\begin{equation*}
\lim_{n\to\infty} E\left| T(x^{\prime }_n) - T(x)\right|^2=
\lim_{n\to\infty} E\left| T_x(g_n) - T_x(e)\right|^2 = 0.
\end{equation*}
This argument shows that every sequence $\{x_n\}$ converging to $x$ in $X$
admits a subsequence $\{x^{\prime }_n\}$ such that $T(x^{\prime }_n)$
converges to $T(x)$ in $L^2(P)$, and this fact is exactly equivalent to
relation (\ref{e:msch}). $\blacksquare$

\bigskip

As already recalled, our findings apply to the important case where $X$
equals the $n$ dimensional unit sphere $S^{n}$, $n\geq 1$, on which the
compact group $SO(n+1)$ acts transitively. Assume for notational simplicity
that $T$ is zero-mean ($ET=0$), and write $\Gamma
(x_{1},x_{2}):=ET(x_{1})T(x_{2})$ for the covariance function of the random
field, $\Gamma :S^{n}\times S^{n}\rightarrow \mathbb{R}$ . By isotropy,
there exists a function $\widetilde{\Gamma }(.):\mathbb{R}^{+}\rightarrow 
\mathbb{R}$ such that $\Gamma (x_{1},x_{2})=\widetilde{\Gamma }(\left\Vert
x_{1}-x_{2}\right\Vert )$, where the symbol $\Vert \cdot \Vert $ stands for
Euclidean norm; it is hence straightforward (and well-known, see for
instance \cite{adlertaylor, Leon, Yad}) that, under isotropy, mean-square
continuity is equivalent to continuity of the function $\widetilde{\Gamma }$
at the origin.

Of course, it is not difficult to figure out rotationally invariant,
positive-definite functions which violate the continuity of $\widetilde{%
\Gamma }$ at the origin: a simple example is provided by $\Gamma
:S^{n}\times S^{n}\rightarrow \mathbb{R}$ such that%
\begin{equation}
\Gamma (x,x)=\widetilde{\Gamma }(0)=1\text{ , }\Gamma (x_{1},x_{2})=%
\widetilde{\Gamma }(\left\Vert x_{1}-x_{2}\right\Vert )=0\text{ , for all }%
x_{1}\neq x_{2}\text{ , }x_{1},x_{2}\in S^{n}\text{ . }  \label{pic}
\end{equation}%
A consequence of Theorem \ref{t:msch} is that such a $\Gamma (.,.)$ cannot
be the covariance function of a measurable isotropic process. More
generally, under isotropy it is immediate to see that continuity of $%
\widetilde{\Gamma }$ at the origin entails continuity everywhere of the
covariance function $\Gamma$, indeed%
\begin{eqnarray*}
\Gamma (x_1,y_{1})-\Gamma (x_2,y_{2})&=&E\left\{
T(x_1)(T(y_{1})-T(y_{2}))\right\} +E\left\{ T(y_2)(T(x_{1})-T(x_{2}))\right\}
\\
&\leq& \sqrt{ET^{2}(x_1)E(T(y_{1})-T(y_{2}))^{2}} + \sqrt{%
ET^{2}(y_2)E(T(x_{1})-T(x_{2}))^{2}},
\end{eqnarray*}
and the last term of the previous chain of inequalities converges to zero,
whenever $(x_1,y_1)\to (x_2,y_2)$ and the isotropic field $T$ is mean-square
continuous (or, equivalently, $\widetilde{\Gamma}$ is continuous at the
origin). We can hence state the following:

\begin{corollary}
The covariance function of a measurable finite-variance isotropic random
field on the homogenous space of a compact group is necessarily everywhere
continuous.
\end{corollary}

\section{Some Historical Remarks}

Our result can be viewed as a characterization of covariance functions for
random fields defined on homogenous spaces of compact groups. In a related
setting, the characterization of covariance functions for stationary and
isotropic random fields in $\mathbb{R}^{d}$ was first considered in a
celebrated paper by Schoenberg (1938), see \cite{Schoenberg}. In this
reference, it was conjectured that the only form of discontinuity which
could be allowed for such covariance functions would occur at the origin,
i.e. given any zero-mean, finite-variance and isotropic random field $Z:%
\mathbb{R}^{d}\rightarrow \mathbb{R},$ its covariance function should be of
the form 
\begin{eqnarray*}
EZ(0,...,0)Z(t_{1},...,t_{d}) &=&\Gamma (\left\Vert
t_{1},...,t_{d}\right\Vert ) \\
&=&\widetilde{\Gamma }(t)=\widetilde{\Gamma }_{0}(t)+\widetilde{\Gamma }%
_{1}(t)\text{ ,}
\end{eqnarray*}%
where as before $t:=\left\Vert t_{1},...,t_{d}\right\Vert $ is Euclidean
norm, and $\widetilde{\Gamma }_{0}(.)$,$\widetilde{\Gamma }_{1}(.):\mathbb{R}%
^{+}\rightarrow \mathbb{R}$ are such that%
\begin{equation*}
\widetilde{\Gamma }_{0}(t)=\left\{ 
\begin{array}{l}
\gamma \geq 0\text{ , for }t=0\text{ ,} \\ 
0\text{ otherwise ,}%
\end{array}%
\right.
\end{equation*}%
while $\widetilde{\Gamma }_{1}(t)$ is nonnegative definite and continuous.
In a later paper which went largely unnoticed, Crum (1956) (\cite{Cru})
proved the conjecture to be right for $d\geq 2.$ This result was drawn to
the attention of the Geostatistics community by Gneiting and Sasvari in 1997
(see \cite{Gneiting}) who argued then that isotropic random fields could be
always expressed as a mean-square continuous component plus a "nugget
effect", e.g. a purely discontinuous component. The fact that this latter
component should be necessarily non-measurable was pointed out (for
instance) in an oral presentation by Starkloff (2009) -- see \cite[p. 13]%
{Starckloff}, as well as Example 1.2.5 in \cite{Kallianpur}. Our results in
this note, though, are obtained in a somewhat different setting (e.g.
homogeneous spaces of compact groups), and we leave for future research a
complete analysis of the relationship between measurability and mean-square
continuity in non-compact circumstances.

\end{document}